 \documentclass[11pt,letterpaper]{amsart}

 \usepackage{txfonts} 

 \usepackage[colorlinks,linkcolor=blue,citecolor=blue]{hyperref}
 
 \theoremstyle{plain}

\newtheorem{theorem}{Theorem}[section]
\newtheorem{lemma}[theorem]{Lemma}

\theoremstyle{definition}
\newtheorem{definition}[theorem]{Definition}

\newtheorem{corollary}[theorem]{Corollary}
\newtheorem{proposition}[theorem]{Proposition}

\newtheorem{problem}[theorem]{Problem}
 
\theoremstyle{remark}
\newtheorem{remark}[theorem]{Remark}

\numberwithin{equation}{section}

\begin{document}

\title[Note on some prescribed problems]
{An extension of prescribed problems on the conformal classes of complete metrics}


\author{Rirong Yuan} \address{School of Mathematics, South China University of Technology, Guangzhou 510641, China}
\email{yuanrr@scut.edu.cn}


\date{}

\dedicatory{}

\begin{abstract}
We solve prescribed problems for modified Schouten tensors in the conformal classes of smooth complete metrics, which extends the results obtained in prequel \cite{yuan-PUE1}. The key ingredient is to confirm the uniform ellipticity of operators under an   assumption, which is sharp as shown by  obstructions from topology and function theory.

\end{abstract}

\maketitle


  
  \section{Introduction}  

  Let $(M,g)$ be a smooth connected   Riemannian manifold of dimension $n$  with Levi-Civita connection $\nabla$. 
Let $\partial M$ denote the boundary of $M$,  $\bar M=M\cup\partial M$.  (Notice $\bar M=M$ if $\partial M=\emptyset$).  
Let  ${Ric}_g$ and   ${R}_g$ denote the  Ricci and scalar curvature of the  metric $g$, respectively. 
For $n\geq3$ we  
denote the modified Schouten tensor by   
\begin{equation}
	\begin{aligned}
		A_{{g}}^{\tau,\alpha}=\frac{\alpha}{n-2} \left({Ric}_{g}-\frac{\tau}{2(n-1)}   {R}_{g}\cdot {g}\right), \,\, \alpha=\pm1,  \mbox{  }\tau \in \mathbb{R}. \nonumber
	\end{aligned}
\end{equation}
When $\tau=\alpha=1$, it is the Schouten tensor 
$$A_g=\frac{1}{n-2} \left({Ric}_g-\frac{1}{2(n-1)}{R_g}\cdot g\right).$$ 

Given a smooth positive 
function $\psi$ on $M$, it is natural to consider the problem of finding
a smooth complete metric $\tilde{g}=e^{2{u}}g$ satisfying
\begin{equation}
	\label{main-equ1}
	\begin{aligned}
		f(\lambda(\tilde{g}^{-1}A_{\tilde{g}}^{\tau,\alpha}))=\psi.
	\end{aligned}
\end{equation} 
Here as in \cite{CNS3}, $f$ is a \textit{smooth},   \textit{symmetric}   function defined in an \textit{open},  \textit{symmetric} and \textit{convex} cone $\Gamma\subset\mathbb{R}^n$ with  vertex at the origin, $\partial \Gamma\neq \emptyset$, and
$$\Gamma_n:=\left\{\lambda=(\lambda_1,\cdots,\lambda_n)\in \mathbb{R}^n: \mbox{ each } \lambda_i>0\right\}\subseteq\Gamma.$$
 We call $\Gamma$ a type 1 cone if $(0,\cdots,0,1)\in\partial\Gamma$; otherwise $\Gamma$ is a type 2 cone.

The problem  was studied extensively. The most important case is the scalar curvature equation
  ($f=\sigma_1$), which is closely related to Yamabe problem proved by   Aubin \cite{Aubin1976},  Schouten \cite{Schoen1984} and Trudinger \cite{Trudinger1968}.
A singular  Yamabe problem was studied by Loewner-Nirenberg \cite{Loewner1974Nirenberg} on spheres, and further extended by  Aviles-McOwen  \cite{Aviles1988McOwen} to general closed Riemannian manifolds. 
In contrast with 
the resolution of
Yamabe problem on closed Riemannian manifolds, 
the 
complete noncompact version of
Yamabe problem 
is not always solvable as shown by
Jin \cite{Jin1988}.  Consequently, 
one could only expect  
the solvability of prescribed curvature problem in the conformal class of complete metrics
under proper additional assumptions. When imposing fairly strong restrictions to the asymptotic ratio of prescribed functions and
curvature 
of background manifolds, 
Aviles-McOwen \cite{Aviles1988McOwen2}  and Jin \cite{Jin1993} investigated prescribed scalar curvature equation on negatively curved complete noncompact Riemannian manifolds.

The fully nonlinear equation \eqref{main-equ1}  on closed manifolds has been investigated by many mathematicians over the past years since the work of \cite{Gursky2003Viaclovsky}, in which the assumption  $\tau<1$, $\alpha=-1$ plays important roles. When the background manifold is compact   with boundary, this is a fully nonlinear version of Loewner-Nirenberg problem. The 
case  $f=\sigma_k^{1/k}$ was studied by Guan \cite{Guan2008IMRN} and Gursky-Streets-Warren \cite{Gursky-Streets-Warren2011} for negative Ricci tensor, and by Li-Sheng \cite{Li2011Sheng}
for $\tau>n-1$, $\alpha=1$. 
When  $(M,g)$ is a complete noncompact manifold, the equation \eqref{main-equ1} with $f=\sigma_k^{1/k}$   was considered   by 
Fu-Sheng-Yuan \cite{FuShengYuan}  for $\alpha=-1$, $\tau<1$,  imposing certain restrictions to the asymptotic ratio of prescribed functions and
curvature 
of background manifolds.

Their results were extended by    \cite{yuan-PUE1}, in which the author introduced  $\kappa_\Gamma$ for $\Gamma $ 
\begin{equation}\label{kappa1-Gamma}	\begin{aligned}	 {\kappa}_{\Gamma}=\max \left\{k: ({\overbrace{0,\cdots,0}^{k-\mathrm{entries}}},
		{\overbrace{1,\cdots, 1}^{(n-k)-\mathrm{entries}}})\in \Gamma \right\}	\end{aligned}\end{equation}
	in order to measure the partial uniform ellipticity of fully nonlinear equations. 
	More precisely, the author solved the prescribed problems in the conformal class of smooth complete metrics under the   assumptions that 
\begin{equation}\label{concave}	\begin{aligned}
		f \text { is concave in } \Gamma,	
\end{aligned}\end{equation} 
\begin{equation}
	\label{homogeneous-1-buchong2}
	\begin{aligned}
		f >0 \mbox{ in }  \Gamma, \,\mbox{ } f  =0 \mbox{ on } \partial\Gamma,
	\end{aligned}
\end{equation}
\begin{equation}
	\label{homogeneous-1-mu}
	\begin{aligned}
		f(t\lambda)=t^\varsigma f(\lambda), \mbox{ } \forall \lambda\in\Gamma, 
		\mbox{  for some } 0<\varsigma \leq1,
	\end{aligned}
\end{equation} 
\begin{equation}
	\label{tau-alpha}
	\begin{cases}
		\tau<1 \,&\mbox{ if } \alpha=-1, \\ 
		\tau>1+ {(n-2)}/{(1+n\kappa_\Gamma\vartheta_{\Gamma})}\,&\mbox{ if } \alpha=1.
	\end{cases}
\end{equation}
Henceforth  
$\vartheta_\Gamma$ is given by 
 \eqref{theta1-gamma} below.

\begin{remark}
	A somewhat surprising fact to us is that the following condition is not required throughout this paper
	\begin{equation}
		\label{elliptic}
		\begin{aligned}
			f_i(\lambda):=\frac{\partial f}{\partial\lambda_i}(\lambda)>0 \mbox{ in } \Gamma, \mbox{ } \forall 1\leq i\leq n. \nonumber
		\end{aligned}
	\end{equation}
\end{remark}

The  subject  of this article is to extend the results.
To this end, for $\Gamma$ we denote $\varrho_\Gamma$  the constant with
\begin{equation}
	\label{varrho1-Gamma}
	\begin{aligned}
		(1,\cdots,1,1-\varrho_\Gamma)\in\partial \Gamma.
	\end{aligned}
\end{equation} 
In this paper we solve the prescribed problems, replacing the condition \eqref{tau-alpha} by 
\begin{equation}
	\label{tau-alpha-sharp}
	\begin{cases}
		\tau<1 \,&\mbox{ if } \alpha=-1,\\ 
		\tau>1+ (n-2)\varrho_\Gamma^{-1} \,&\mbox{ if } \alpha=1.
	\end{cases}
\end{equation} 
 The case $\tau=1$, $\alpha=-1$ is   allowed when $\Gamma$ is of type 2. Moreover, 
this is the optimal condition 
in the sense that the equation \eqref{main-equ1} is uniformly elliptic at any admissible metrics. 

\begin{remark}
	\label{RK1-1}
	 For $\Gamma=\Gamma_k$,  
	 $\varrho_{\Gamma_k}=\frac{n}{k}$. Moreover,  using Theorem \ref{yuan-k+1} we can check that $$\varrho_\Gamma\geq 1+n\kappa_\Gamma\vartheta_\Gamma\geq (1-\kappa_\Gamma\vartheta_\Gamma)^{-1}.$$
	 So the condition \eqref{tau-alpha-sharp} is broader than \eqref{tau-alpha}.
\end{remark}


The notions of admissible,  pseudo-admissible,  quasi-admissible, and maximal metrics are given  as in  the following:

\begin{definition}
	\label{def1-admissible}
	
	For the equation \eqref{main-equ1}, we say $g$ is \textit{admissible}   if   \begin{equation}\begin{aligned}\lambda(g^{-1}A_{{g}}^{\tau,\alpha})\in\Gamma \mbox{ in  }\bar M.  \nonumber
	\end{aligned}    \end{equation}  
	Similarly, we call $g$ a \textit{pseudo-admissible} metric if
	\begin{equation}\begin{aligned}
			\lambda(g^{-1}A_{g}^{\tau,\alpha}) \in \bar{\Gamma}   \mbox{ in } 
			\bar M. \nonumber
	\end{aligned}    \end{equation} 
	Meanwhile, 
	we say a metric $g$ is \textit{quasi-admissible} if  
	\begin{equation}\label{assump1-metric}\begin{aligned}	\lambda(g^{-1}A_{g}^{\tau,\alpha}) \in \bar{\Gamma} \mbox{ in } \bar M, 	
			\mbox{ and }
					\lambda(g^{-1}A_{g}^{\tau,\alpha}) \in \Gamma \mbox{ at some } p_0\in \bar M. \nonumber
			\end{aligned}\end{equation}
			Moreover, we say $\tilde{g}=e^{2u}g$ is a \textit{maximal} admissible metric to the equation \eqref{main-equ1}, 
			if
			$u\geq w$ in $M$ for any admissible conformal metric $e^{2w}g$ satisfying the same equation.

			Accordingly,  we have notions of \textit{admissible, pseudo-admissible, quasi-admissible} and \textit{maximal} solutions for more general fully nonlinear elliptic 	
			equations.
		\end{definition} 
	
	Building on Proposition \ref{key-lemma1}
	via partial uniform ellipticity, 
	together with the results 
	in prequel \cite{yuan-PUE1}, we obtain the following results.



The first result is to
 solve 
\eqref{main-equ1} on closed 
manifolds with quasi-admissible metric.
\begin{theorem}
	\label{thm2-metric}
Suppose $(f,\Gamma)$ satisfies \eqref{concave}, \eqref{homogeneous-1-buchong2}, \eqref{homogeneous-1-mu}, and assume
	$(\alpha,\tau)$ satisfies \eqref{tau-alpha-sharp}. 
	Assume that $M$ is a closed connected   manifold of dimension $n\geq 3$ 
	with a $C^2$ quasi-admissible Riemannian metric $g$.
	Then for any $0<\psi\in C^\infty(M)$,  there is a unique smooth conformal admissible metric   $\tilde{g} $ satisfying
	\eqref{main-equ1}.
\end{theorem}

  
  Next we solve a fully nonlinear version of the Loewner-Nirenberg problem.

 \begin{theorem}
	\label{existence1-compact}
	In addition to \eqref{concave}, \eqref{homogeneous-1-buchong2}, \eqref{homogeneous-1-mu} and  
	\eqref{tau-alpha-sharp}, we assume  
	$(M,g)$ is a compact  connected Riemannian manifold of dimension $n\geq 3$ with smooth boundary and supports a $C^2$  compact pseudo-admissible conformal metric.
	Then for any $0<\psi\in C^\infty(\bar M)$, there exists at least one smooth complete  
	metric $\tilde{g}=e^{2u}g$  
	satisfying \eqref{main-equ1}.  
\end{theorem}

The assumption of pseudo-admissible metric can be further removed in some context.
\begin{theorem} 
	\label{existence1-compact-construction}
	In Theorem \ref{existence1-compact} the assumption on the existence of a compact  pseudo-admissible conformal metric can be removed 
	if $(\alpha,\tau)$ further satisfies 
		\begin{equation}
		\label{tau-alpha-4}
		\begin{cases}
			\tau\leq  2-\frac{2}{\varrho_\Gamma} \,& \mbox{ if } \alpha=-1,\\
			\tau\geq  2  \,& \mbox{ if } \alpha=1. 
		\end{cases}
	\end{equation}
	
\end{theorem}

\begin{remark}
  In the case  $\kappa_\Gamma\leq n-3$,   we may verify  condition \eqref{tau-alpha-4} for $\alpha=1$ under the assumption \eqref{tau-alpha-sharp}. See Corollary \ref{coro1-verify}.
  \end{remark}

Finally we address a   fully nonlinear version of Yamabe problem 
 on a complete noncompact 
manifold, under the asymptotic assumption at infinity:
 There exists a \textit{complete pseudo-admissible} 
metric $\underline{g}=e^{2\underline{u}}g$ 
with 
\begin{equation}
	\label{key-assum1}
	\begin{aligned}
		{f(\lambda(\underline{g}^{-1}A_{\underline{g}}^{\tau,\alpha}))}   \geq  \Lambda_0 {\psi} \mbox{ holds uniformly in $M\setminus K_0$.}
	\end{aligned}
\end{equation}
Here $\Lambda_0$ is a uniform positive constant, and $K_0$ is a compact subset of $M$.

\begin{theorem}
	\label{thm1}
	Let $(M,g)$ be a complete noncompact Riemannian manifold of dimension $n\geq 3$. 
	Suppose $(f,\Gamma)$ satisfies   \eqref{concave}, \eqref{homogeneous-1-buchong2} and \eqref{homogeneous-1-mu}.
	Let $(\alpha,\tau)$ obey  \eqref{tau-alpha-sharp}.
	Given a positive smooth function $\psi$, we assume that
	$(M,g)$ carries a  $C^2$ complete pseudo-admissible conformal metric subject to \eqref{key-assum1}.
	Then there  
	exists 
	a unique smooth complete maximal conformal  admissible metric 
	satisfying 
	\eqref{main-equ1}. 
\end{theorem}

The paper is organized as follows. In Section \ref{sec2-prelimi} we summarize some notations, formulas, and results on Morse functions and partial uniform ellipticity. Also we present some existence results of fully nonlinear uniformly elliptic equations for Schouten tensor.
In Section \ref{sec3-proof} we complete the proof by constructing operators of uniform ellipticity under the assumption \eqref{tau-alpha-sharp}. 
In Section \ref{sec1-topo} we present some obstructions 
to indicate  that the assumption on $\tau$ we impose is sharp. 
In Section \ref{sec1-cones} we briefly discuss the cones and then prove some properties for  $\varrho_\Gamma$ and $\kappa_\Gamma$.

\medskip

\section{Preliminaries}
\label{sec2-prelimi}

\subsection{Notations and Formulas}

Given a cone $\Gamma$, we denote $\Gamma_\infty$ 
 the projection of $\Gamma$ to the subspace of former $n-1$ subscripts, that is   \begin{equation}
	\label{construct1-Gamma-infty}
	\begin{aligned}
		\Gamma_\infty:=\{(\lambda_{1}, \cdots, \lambda_{n-1}): (\lambda_{1}, \cdots, \lambda_{n-1},\lambda_n) \in \Gamma\}.
	\end{aligned}
\end{equation}

Let $e_1,...,e_n$ be a local frame on $M$.  
Denote
$$  \langle X,Y\rangle=g(X,Y),\,\  g_{ij}= \langle e_i,e_j\rangle,\,\   \{g^{ij} \} =  \{g_{ij} \}^{-1}.$$
Under Levi-Civita connection  of $(M,g)$, $\nabla_{e_i}e_j=\Gamma_{ij}^k e_k$, and $\Gamma_{ij}^k$ denote the Christoffel symbols.  
For simplicity we write 
$$\nabla_i=\nabla_{e_i}, \nabla_{ij}=\nabla_i\nabla_j-\Gamma_{ij}^k\nabla_k, 
\nabla_{ijk}=\nabla_i\nabla_{jk}-\Gamma_{ij}^l\nabla_{lk}-\Gamma^l_{ik}\nabla_{jl}, \mbox{ etc}.$$

Under the conformal change
$\tilde{g}=e^{2u}g$, one has  
 (see e.g. \cite{Besse1987} or \cite{Gursky2003Viaclovsky})
\begin{equation}
\label{conformal-formula1}
\begin{aligned}
A_{\tilde{g}}^{\tau,\alpha}
= A_{g}^{\tau,\alpha}
+\frac{\alpha(\tau-1)}{n-2}\Delta u g-\alpha  \nabla^2 u
+\frac{\alpha(\tau-2)}{2}|\nabla u|^2 g
+\alpha  du\otimes du. 
\end{aligned}
\end{equation}
 In particular, the Schouten tensor obeys
 \begin{equation}\label{conformal-formula2}\begin{aligned}A_{\tilde{g}} = A_{g} -\nabla^2 u	-\frac{1}{2}|\nabla u|^2 g+ du\otimes du. \end{aligned}\end{equation}
Throughout this paper, $\Delta u$, $\nabla^2 u$ and $\nabla u$ are the Laplacian, Hessian and gradient of $u$ with respect to $g$, respectively.
For simplicity,  we denote
\begin{equation}
\label{beta-gamma-A2}
\begin{aligned}
V[u]=\Delta u g -\varrho\nabla^2 u+\gamma |\nabla u|^2 g +\varrho du\otimes du+A,   
\end{aligned} 
\end{equation} 
\begin{equation}
\label{beta-gamma-A-3}
\begin{aligned}
\varrho=\frac{n-2}{\tau-1}, \mbox{ }
\gamma=\frac{(\tau-2)(n-2)}{2(\tau-1)}, \mbox{ } 
A=\frac{n-2}{\alpha(\tau-1)} A_{g}^{\tau,\alpha}.  
\end{aligned}
\end{equation}
Here  $V[u]=\frac{n-2}{\alpha(\tau-1)} A^{\tau,\alpha}_{\tilde{g}}, \mbox{ } \tilde{g}=e^{2u}g$.
Moreover, by the strightforward computation
\begin{equation}
	\label{computation1-add}
	\begin{aligned}
		V[\underline{u}+w]=\,&
		V[w]+\Delta \underline{u} g -\varrho\nabla^2 \underline{u}+\gamma |\nabla \underline{u}|^2 g +\varrho d\underline{u}\otimes d\underline{u} \\
		\,&+ 2\gamma  \langle \nabla w,\nabla \underline{u}\rangle g+\varrho (d\underline{u}\otimes dw+dw\otimes d\underline{u}). 
	\end{aligned}
\end{equation}

We summarize more useful notion. 
Let $\mathrm{Sec}_g$ stand for the sectional curvature of $g$, let
$G_g$ be the Einstein tensor 
$$G_g=Ric_g-\frac{R_g}{2} g.$$

In dimension three the Einstein tensor is closely related to sectional curvature.  
\begin{proposition}
	\label{prop1-einstein-sectional}
	Fix $x\in M^3$,  let $\Sigma\subset T_xM$ be a tangent $2$-plane,   $\vec{\bf n}\in T_xM$ the unit normal vector to $\Sigma$, then  
	\begin{equation} \label{sectional-einstein} \begin{aligned}
			G_g(\vec{\bf n},\vec{\bf n})=-\mathrm{Sec}_g(\Sigma). 
	\end{aligned}  \end{equation} 
\end{proposition}
\begin{remark}
	This formula motivates 
	Gursky-Streets-Warren  \cite{Gursky-Streets-Warren2010}  
	to prove that any Riemannian 3-manifold  with smooth boundary admits a complete conformal metric of ``almost negative"  
	curvature, via 
	a Monge-Amp\`ere type equation. 
\end{remark}

\subsection{Some result on Morse function}

The following lemma asserts that any compact manifold with boundary carries a Morse function without any critical point.    

\begin{lemma}
	\label{lemma-diff-topologuy}
	Let $(M,g)$
	be a compact connected Riemannian 
	manifold of dimension $n\geq 2$ with smooth boundary. Then there is a smooth function $v$ without any critical points, that is $d v\neq 0$.
\end{lemma}

\begin{proof} 
	The construction  is more or less standard in differential topology.
	Let $X$ be the double of $M$. Let $w$ be a smooth Morse function on $X$ with the critical set $\{p_i\}_{i=1}^{m+k}$, among which $p_1,\cdots, p_m$ are all the critical points  being in $\bar M$. 
	Pick $q_1, \cdots, q_m\in X\setminus \bar M$ but not the critical point of $w$. By homogeneity lemma 
	(see \cite{Milnor-1997}), 
	one can find a diffeomorphism
	$h: X\to X$, which is smoothly isotopic to the identity, such that 
	\begin{itemize}
		\item $h(p_i)=q_i$, $1\leq  i\leq  m$.
		\item $h(p_i)=p_i$, $m+1\leq  i\leq  m+k$.
	\end{itemize}
	Then $v=w\circ h^{-1}\big|_{\bar M}$ is the desired 
	function.
	
\end{proof}

\subsection{On the nonlinear operators}

 Let $\kappa_\Gamma$ be as in \eqref{kappa1-Gamma}. 
As in \cite{yuan-PUE1} for $\Gamma$ we introduce 
$\vartheta_\Gamma$  as in the following:
\begin{equation}
	\label{theta1-gamma}
	\vartheta_\Gamma=
	\begin{cases}  
		1/n, \,& \Gamma=\Gamma_n,\\
\underset{(-\alpha_1,\cdots,-\alpha_{\kappa_\Gamma}, \alpha_{\kappa_\Gamma+1},\cdots, \alpha_n)\in \Gamma;\mbox{ } \alpha_i>0}{	\sup}\frac{\alpha_1/n}{\sum_{i=\kappa_\Gamma+1}^n \alpha_i-\sum_{i=2}^{\kappa_\Gamma}\alpha_i}, \,& \Gamma\neq\Gamma_n. 
	\end{cases} 
\end{equation}

 For any concave and symmetric function $f$ obeying 
   \begin{equation}	
 	\label{addistruc}	
 	\begin{aligned}
 	 \lim_{t\rightarrow+\infty} f(t\lambda)>f(\mu), \,\, \forall \lambda,\mbox{ }  \mu\in\Gamma, 
 	\end{aligned}
 \end{equation}
the author  \cite{yuan-PUE1}   proved the following result concerning partial uniform ellipticity.
\begin{theorem} [\cite{yuan-PUE1}]
	\label{yuan-k+1}
	Suppose \eqref{concave} and \eqref{addistruc} hold. Then for any $\lambda\in \Gamma$ with
	$\lambda_1 \leq \cdots \leq\lambda_n$,
	\begin{enumerate}
		\item[$(\bf 1)$]   $f_i(\lambda)\geq 0,   \mbox{ } \forall 1\leq i\leq n, $ 
		$\sum_{i=1}^n f_i(\lambda)>0.$
		\item[$(\bf 2)$]  $f_{{i}}(\lambda) \geq n    \vartheta_{\Gamma}f_1(\lambda) \geq \vartheta_{\Gamma} \sum_{j=1}^{n}f_j(\lambda),  \mbox{ } \forall  1\leq i\leq 1+ 	\kappa_{\Gamma}.$
	\end{enumerate}
 Moreover, the assertion of $(\kappa_\Gamma+1)$-uniform ellipticity
 cannot be improved.
	
\end{theorem} 





In particular, we have the following conclusion.
\begin{lemma} [\cite{yuan-PUE1}]
\label{lemma5.11}
Suppose $(f,\Gamma)$ obeys  \eqref{concave} and
\eqref{addistruc}.  
Then the following 
are equivalent: 
\begin{itemize}
	\item $\Gamma$ is of type 2. 
	 That is $\Gamma_\infty=\mathbb{R}^{n-1}$. 
	\item  There is a uniform constant $\theta$ such that 
	\begin{equation}
		\label{fully-uniform2}
		\begin{aligned}
			f_{i}(\lambda)\geq  \theta\sum_{j=1}^n f_j(\lambda)>0 	\mbox{ in } \Gamma,	\,\, \forall 1\leq i\leq  n.
		\end{aligned}
	\end{equation}
\end{itemize}
\end{lemma}
 
 We may verify \eqref{addistruc} in the assumptions imposed in main results.
   \begin{lemma}[\cite{yuan-PUE1}]
 	\label{lemma2.3}
 	
 	Assume,  in addition to \eqref{concave}, that 
 	$\sup_\Gamma f=+\infty$ and
 	\begin{equation}
 		\label{addistruc-0}
 		\begin{aligned}
 			\lim_{t\rightarrow+\infty} f(t\lambda)>-\infty, \mbox{  } \forall\lambda\in\Gamma.
 		\end{aligned}
 	\end{equation}
 	Then $(f,\Gamma)$ satisfies \eqref{addistruc}.
 \end{lemma}

\subsection{Some geometric conclusions}
\label{subsec1-geometric-conclusion}

In \cite{yuan-PUE1} the author proved the following results.

\begin{theorem}[\cite{yuan-PUE1}]
	\label{existence1-compact-2}
	Suppose $f$  satisfies \eqref{concave}, \eqref{homogeneous-1-buchong2}, \eqref{homogeneous-1-mu}  and 
	\eqref{fully-uniform2} in $\Gamma$.
	Let $(M,g)$ be a compact connected Riemannian manifold of dimension $n\geq  3$ with smooth boundary and support a $C^2$  
	conformal metric 
	with $\lambda(-g^{-1}A_{\underline{g}})\in \Gamma$ in $\bar M$.
	Then for any $0<\psi\in C^\infty(\bar M)$, there exists at least one smooth  
	complete  
	metric $\tilde{g}=e^{2u}g$ 
	satisfying 
	${f}(\lambda(-\tilde{g}^{-1}A_{\tilde{g}}))= \psi$
	and  $\lambda(-{g}^{-1}A_{\tilde{g}})\in \Gamma$
	in $M$.
	
		\end{theorem}

		\begin{theorem}[\cite{yuan-PUE1}]
			\label{thm1-shouten}
			Suppose $(f,\Gamma)$  satisfies \eqref{concave}, \eqref{homogeneous-1-buchong2}, \eqref{homogeneous-1-mu} and 
			\eqref{fully-uniform2}.
			Let $(M,g)$ be a complete noncompact Riemannian manifold of dimension $n\geq 3$ and with a $C^2$   complete  conformal  metric $\underline{g}$ subject to 
			\begin{equation}
				\label{key-assum2-2}
				\begin{aligned}
					f(\lambda(-\underline{g}^{-1} A_{\underline{g}})) \geq  \Lambda_1  \psi, \,\, \lambda(-g^{-1}A_{\underline{g}})\in \Gamma,  \mbox{ } \mbox{ in } M  
				\end{aligned}
			\end{equation}   
			where $0<\psi\in C^\infty(M)$ and $\Lambda_1$ is a uniform positive constant.
			Then there  
			exists 
			a unique 
			smooth   maximal   complete  metric $\tilde{g}=e^{2u}g$ satisfying 
			$f(\lambda(-\tilde{g}^{-1} A_{\tilde{g}})) =\psi$ and $\lambda(- {g}^{-1} A_{\tilde{g}}) \in\Gamma.$
		\end{theorem}

 \medskip

\section{Proof of main results}
\label{sec3-proof}

\subsection{Analytic optimal condition and construction of uniformly elliptic operators}
\label{sec1-AOC}

 
 From \eqref{conformal-formula1}   equation  \eqref{main-equ1} ($\tau\neq1$) can be reduced to an equation  of the form
\begin{equation}
	\label{equ0-0} 
	\begin{aligned} 
			{f}(\lambda[g^{-1}(\Delta u  g -\varrho\nabla^2 u+A(x,u,\nabla u))])=\psi(x,u,\nabla u).  	\end{aligned}
\end{equation} 
We draw the optimal analytic condition  under which 
the equation \eqref{equ0-0} is of uniform ellipticity. 

 Given a cone $\Gamma$, we have $\varrho_\Gamma$ that is defined in \eqref{varrho1-Gamma}. Take  a constant $\varrho$ satisfying
 \begin{equation}
 	\label{assumption-4}
 	\begin{aligned} 
 		\varrho<  \varrho_{\Gamma} \mbox{ and } \varrho\neq 0.
 	\end{aligned}
 \end{equation}
So $\varrho<n$. We let  
\begin{equation}	 	\begin{aligned}  		\mu_i=\frac{1}{n-\varrho}
		\left(\sum_{j=1}^n \lambda_j -\varrho\lambda_i \right), 
		\mbox{ i.e.  }
		\lambda_i=\frac{1}{\varrho}\left(\sum_{j=1}^n \mu_j-(n-\varrho)\mu_i\right),	\end{aligned}  \end{equation} 
\begin{equation}
	\label{map1}
	\begin{aligned}
		\tilde{\Gamma}  =
		\left\{(\lambda_1,\cdots,\lambda_n): 
		\lambda_i=\frac{1}{\varrho}\left(\sum_{j=1}^n \mu_j-(n-\varrho)\mu_i\right),
		\mbox{ } (\mu_1,\cdots,\mu_n)\in \Gamma 
		\right\}.
	\end{aligned}
\end{equation} 
 We can check that $\tilde{\Gamma}$ is also an open symmetric convex cone in $\mathbb{R}^n$. Note  that $\sum_{i=1}^n\lambda_i=\sum_{i=1}^n\mu_i$,  we know $\tilde{\Gamma}\subseteq\Gamma_1$. For any $\lambda\in\tilde{\Gamma}$ there exists a unique $\mu\in\Gamma$ such that
$$\lambda_i=\frac{1}{\varrho}\left(\sum_{j=1}^n \mu_j-(n-\varrho)\mu_i\right).$$
One has a symmetric concave function $\tilde{f}$ on  $\tilde{\Gamma}$
as follows:
\begin{equation}\label{def-f}\begin{aligned}  \tilde{f}(\lambda)= f(\mu). 
\end{aligned}\end{equation} 
 
In practice, 
Theorem \ref{yuan-k+1} (or 
Lemma \ref{lemma5.11})
gives the following 
analytic optimal condition, under which 
\eqref{equ0-0}  is uniformly elliptic at each admissible functions.

\begin{proposition} 
	\label{key-lemma1}
Let $(f,\Gamma)$ satisfy \eqref{concave} and \eqref{addistruc}. Given a constant $\varrho$ satisfying \eqref{assumption-4},
	we define $(\tilde{f},\tilde{\Gamma})$  as  in \eqref{map1} and \eqref{def-f}. 
	Then  
		$\tilde{f}$ is of uniform ellipticity in $\tilde{\Gamma}$, that is 
		there exists a uniform positive constant $\theta$ such that 
		\begin{equation}	 
			\label{FUE-1}
			\begin{aligned}
				\frac{\partial \tilde{f}}{\partial \lambda_i}(\lambda) \geq \theta \sum_{j=1}^n \frac{\partial \tilde{f}}{\partial \lambda_j}(\lambda) >0, \,\, \forall \lambda\in\tilde{\Gamma},   \mbox{ } \forall 1\leq i\leq n. \nonumber
			\end{aligned}
		\end{equation}
 
\end{proposition}

\begin{proof}
	Since $\varrho<\varrho_\Gamma$, we get $(1,\cdots,1,1-\varrho)\in\Gamma$. This is equivalent to $(0,\cdots,0,1)\in\tilde{\Gamma}$, as required.
\end{proof}

As a result, 
one can  reduce  prescribed curvature equation \eqref{main-equ1} 
to a uniformly elliptic equation for conformal deformation of Schouten tensor. To do this, let $(\tilde{f},\tilde{\Gamma})$  be as in \eqref{map1}-\eqref{def-f},  let 
$\varrho=\frac{n-2}{\tau-1}$ be as in \eqref{beta-gamma-A-3}.
We can check  
\begin{equation}
	\label{check-1}
	\begin{aligned}
		\mathrm{tr}\left(g^{-1}(-A_g)\right)g
		-\varrho \left(-A_g\right)
		=  \frac{n-2}{\alpha(\tau-1)}A_{g}^{\tau,\alpha}.
		\nonumber
	\end{aligned}
\end{equation}
Thus,  when $f$ is  homogeneous of degree $\mathrm{\varsigma}$, 
equation
\eqref{main-equ1} is equivalent to 
\begin{equation}
	\label{mainequ-02-0}
	\begin{aligned}
  \tilde{f}(\lambda(-{g}^{-1}A_{\tilde{g}})) = 
		\left( \frac{n-2}{\alpha(n\tau+2-2n)}\right)^\mathrm{\varsigma} \psi e^{2\mathrm{\varsigma} u},
	\end{aligned}
\end{equation} 
in which $\tilde{f}$ is of fully uniform ellipticity in $\tilde{\Gamma}$, according to Proposition \ref{key-lemma1}. 

 \begin{remark}
 	Let $\varrho_\Gamma$ be as in \eqref{varrho1-Gamma}.
	Clearly $1\leq \varrho_\Gamma\leq n$. In addition,
	\begin{itemize}
		\item  $\varrho_\Gamma=1$ if and only if $\Gamma=\Gamma_n$.
		\item 
		$\varrho_{\Gamma}=n$ if and only if $\Gamma=\Gamma_1$.
	\end{itemize}
	
	
\end{remark}

\begin{remark}
	\label{remark-3.3}
	Due to  the obstructions to the existence of  complete metrics with (uniformly)
	 positive scalar curvature, one could not expect the solvability of 
	\begin{equation}
		\label{equ2-Schouten}
		\begin{aligned}
			\tilde{f}(\lambda(\tilde{g}^{-1} A_{\tilde{g}})) =\psi  
		\end{aligned}
	\end{equation} in the conformal class of complete admissible metrics. 
	On the other hand, 
	the Dirichlet problem for \eqref{equ2-Schouten} was solved by Guan \cite{Guan2007AJM}, given an admissible subsolution.
\end{remark}

\subsection{Construction of admissible metrics}
\label{subsec1-construction-admissiblemetric}

According to Proposition \ref{key-lemma1}, 
 Theorems \ref{existence1-compact-2} and \ref{thm1-shouten}, it requires only to construct admissible conformal metrics.
The constructions are given in \cite{yuan-PUE1}. We present them here for completeness.

Given a $C^2$-smooth  function $w$ on $M$, we 
denote the critical set by 
$\mathcal{C}(w)=\{x\in  M: dw(x)=0\}.$
Also we use the notation denoted in \eqref{beta-gamma-A2} and \eqref{beta-gamma-A-3}.

\begin{proposition}
	\label{thm2-construction}
	Let $(\alpha,\tau)$ satisfy 
	\eqref{tau-alpha-sharp}.
	Assume $M$ is a closed connected manifold of dimension $n\geq 3$ and suppose a  $C^2$ {\em quasi-admissible} Riemannian metric $g$.  
	Then there exists an admissible metric being conformal to $g$.
\end{proposition}

\begin{proof}

	By 
	the openness of $\Gamma$ and the assumption 
	of quasi-admissible metric, there exists a uniform positive constant $r_0$
	such that  
	\begin{equation}
		\label{key1-metric}
		\begin{aligned}
			\lambda(g^{-1}A)\in \Gamma \mbox{ in } \overline{B_{r_0}(p_0)}.
		\end{aligned}
	\end{equation}
	
	Take a smooth Morse function $w$ with
	the critical set  
	$$\mathcal{C}(w)=\{p_1,\cdots, p_m, p_{m+1}\cdots p_{m+k}\}$$
	among which $p_1,\cdots, p_m$ are all the critical points  being in $M\setminus \overline{B_{r_0/2}(p_0)}$. 
	Pick $q_1, \cdots, q_m\in {B_{r_0/2}(p_0)}$ but not the critical point of $w$. By the homogeneity lemma,
	one can find a diffeomorphism
	$h: M\to M$, which is smoothly isotopic to the identity, such that 
	\begin{itemize}
		\item $h(p_i)=q_i$, $1\leq i\leq m$.
		\item $h(p_i)=p_i$, $m+1\leq i\leq m+k$.
	\end{itemize}
	Then we obtain a   Morse function 
	\begin{equation}
		\label{Morse1-construction}
		\begin{aligned} 
			v=w\circ h^{-1}.
		\end{aligned}
	\end{equation} 
	One can check that 
	\begin{equation}
		\label{key2}
		\begin{aligned}
			\mathcal{C}(v)=\{q_1,\cdots, q_m, p_{m+1}\cdots p_{m+k}\}\subset \overline{B_{r_0/2}(p_0)}.
		\end{aligned}
	\end{equation}
	
	Next we complete the proof. Assume $v\leq -1$.   	Take $\underline{u}=e^{Nv},$ $\underline{g}=e^{2\underline{u}}g,$
	then \begin{equation}
		\label{key3-2}
		\begin{aligned}
			V[\underline{u}]= A+
			N^2e^{Nv}\left((\Delta v g-\varrho\nabla^2v)/N+(1+\gamma e^{Nv})|\nabla v|^2 g+\varrho (e^{Nv}-1)dv\otimes dv\right). 
		\end{aligned}
	\end{equation}

	Notice that  
	\begin{equation}
		\label{key3}
		\begin{aligned}
			\,&\lambda(g^{-1} ((1+\gamma e^{Nv})|\nabla v|^2 g+\varrho (e^{Nv}-1)dv\otimes dv ))\\
			=\,& |\nabla v|^2  \left[(1, \cdots, 1, 1-\varrho)+  e^{Nv}  (\gamma,\cdots,\gamma,\gamma+\varrho)\right].
		\end{aligned}
	\end{equation} 
  Thus
	\begin{equation}
		\label{key4}
		\begin{aligned}
			(1, \cdots, 1, 1-\varrho)+e^{Nv}(\gamma,\cdots,\gamma,\gamma+\varrho)\in \Gamma  \mbox{ for  } N\gg1.
		\end{aligned}
	\end{equation}
	(Notice $v\leq -1$ and $\Gamma$ is open).
	
	\noindent{\bf Case 1}: $x\in \overline{B_{r_0}(p_0)}$. By \eqref{key1-metric} and the openness of $\Gamma$, 
	\begin{equation}
		\begin{aligned}
			\lambda(g^{-1}(A+
			N e^{Nv} (\Delta v g-\varrho\nabla^2v)) )\in \Gamma \mbox{ in } \overline{B_{r_0}(p_0)}.
		\end{aligned}
	\end{equation}
	Combining  \eqref{key3} and \eqref{key4}, 
	\begin{equation}
		\begin{aligned}
			\lambda(g^{-1}V[\underline{u}])\in\Gamma \mbox{ in } \overline{B_{r_0}(p_0)}.  \nonumber
		\end{aligned}
	\end{equation}
	
	\noindent{\bf Case 2}:  $x\notin \overline{B_{r_0}(p_0)}$. By \eqref{key2} there is a uniform positive constant $m_0$ such that $|\nabla v|^2\geq m_0$ in $M\setminus\overline{B_{r_0}(p_0)}$. By \eqref{key3}, \eqref{key4}, 
	and the openness of $\Gamma$, as well as the existence of quasi-admissible metric 
	\begin{equation}
		\begin{aligned}
			\lambda(g^{-1}V[\underline{u}])\in\Gamma \mbox{ in } M\setminus\overline{B_{r_0}(p_0)}.  \nonumber
		\end{aligned}
	\end{equation}
	This completes the proof.
\end{proof}


Based on Lemma \ref{lemma-diff-topologuy}, we have an analogue of Proposition \ref{thm2-construction} for manifolds with boundary.
\begin{proposition}
	\label{thm3-construction}
	Let $(\alpha,\tau)$ satisfy 
	\eqref{tau-alpha-sharp}.
	Suppose $(M,g)$ is a compact connected Riemannian manifold with smooth boundary and carries a $C^2$ {\em pseudo-admissible} conformal metric. 
	Then there is an admissible conformal metric.
\end{proposition}



Next, we construct a complete noncompact admissible metric satisfying an asymptotic property.
\begin{proposition}\label{thm4-construction} 
	Let   $(f,\Gamma)$ and $(\alpha,\tau)$ be as in Theorem \ref{thm1}. 
	Suppose $(M,g)$ is a complete noncompact Riemannian manifold.  Given a positive smooth function $\psi$, we assume $(M,g)$ carries a $C^2$  \textit{complete pseudo-admissible} conformal metric subject to \eqref{key-assum1}. Then there is a  complete noncompact admissible conformal metric $\hat{g}$ with  
	\begin{equation}
		\label{key-assum1-2}
		\begin{aligned}
			{f(\lambda(\hat{g}^{-1}A_{\hat{g}}^{\tau,\alpha}))}   \geq \Lambda_1 {\psi} \mbox{ in } M 
		\end{aligned}
	\end{equation}
	for some positive constant $\Lambda_1$. Moreover, $\hat{u}\geq \underline{u}-C_0$ for some constant $C_0>0$, where  $\underline{u}$ is as in \eqref{key-assum1}.
\end{proposition}

\begin{proof}
	Without loss of generality, by \eqref{computation1-add} one only consider the case $\underline{u}=0$. 
	Let $K_0$ be  as in \eqref{key-assum1} the compact subset.  
	From \eqref{key-assum1}, \eqref{homogeneous-1-buchong2} and the positivity of $\psi$, we know $g$ is admissible when restricted to $M\setminus K_0$.  
	
	Pick  two $n$-dimensional compact 
	submanifolds $M_1$, $M_2$   with smooth boundary  and with $K_0\subset\subset M_1\subset\subset M_2$. 
	Let $v$ be a smooth 
	function  with $dv\neq0$ and 
	$v\leq 0$ 
	on $\bar M_2$ (by Lemma \ref{lemma-diff-topologuy}). Similar to the proof of Proposition \ref{thm2-construction}, 
	as in Proposition \ref{thm3-construction}, 
	$e^{2\underline{w}}g$ 
	is 
	admissible on $\bar M_2$ if  $\underline{w}=e^{t(v-1)}$, $t\gg1$. 
	Choose a cutoff function satisfying
	\begin{equation}
		\begin{aligned}
			\zeta\in C^{\infty}_0(M_2), \, 0\leq\zeta\leq 1 \mbox{ and } \zeta\Big|_{M_1}=1. \nonumber
		\end{aligned}
	\end{equation}   
	Take $\hat{g}=e^{2\hat{u}}g$ where $\hat{u}=e^{Nh}$,
	\begin{equation}
		h=
		\begin{cases}
			\zeta  v-1\,& \mbox{ if } x\in M_2,\\
			-1 \,& \mbox{ otherwise.}\nonumber
		\end{cases}
	\end{equation} 
	Notice $g$ is admissible in $M\setminus K_0$, and $\hat{g}$ is admissible when restricted to $M_1\cup (M\setminus M_2)$.
	Similar to the proof of Proposition \ref{thm2-construction}, we can check for $N\gg1$ that $\hat{g}$ is an  admissible metric and also satisfies \eqref{key-assum1-2}.
	
\end{proof}

The pseudo-admissible metric assumption imposed in Proposition \ref{thm3-construction} can be further removed  in some cases.
\begin{proposition}
\label{lemma5-main}
Let $(M,g)$ be a compact connected Riemannian manifold with smooth boundary.
Given $(\alpha,\tau)$ satisfying 
\eqref{tau-alpha-sharp}.
Let  $\gamma$ and $\varrho$ be as in  \eqref{beta-gamma-A-3}. Assume  
\begin{equation}
	\label{key-condition1-construct}
	(\gamma, \cdots,\gamma,\gamma+\varrho)\in\bar{\Gamma}.
\end{equation}
Then 
there exists a smooth conformal  admissible metric  
on $\bar M$.
\end{proposition}

\begin{proof}
According to Lemma \ref{lemma-diff-topologuy}, there exists a smooth   function $v$ with  $v\geq0$ and
\[ |\nabla v|^2\geq a_0>0  \mbox{ in } \bar M \]
for some positive constant $a_0$.  
Set $ \underline{u}=e^{Nv}$.  
By  \eqref{key3-2} and \eqref{key3}, $\lambda(g^{-1}V[\underline{u}])\in \Gamma$ in $\bar M$ for $N\gg1$. Thus $\underline{g}=e^{2\underline{u}}g$ is 
an admissible metric. 

\end{proof}

\subsubsection*{Confirm the condition \eqref{key-condition1-construct}}
Next we confirm \eqref{key-condition1-construct} in some cases.
By a simple computation
	\begin{equation}
	\label{gammarho0-1}
	\gamma =\frac{(n-2)(\tau-2)}{2(\tau-1)},\,\  \gamma+\varrho
	=\frac{\tau(n-2)}{2(\tau-1)}. \nonumber
\end{equation}
Thus under the assumption 
\begin{equation}
	\label{tau-alpha-3}
	\begin{cases}
		\tau\leq  0  \,& \mbox{ if } \alpha=-1,\\
		\tau\geq  2  \,& \mbox{ if } \alpha=1,
	\end{cases}
\end{equation} 
we obtain $\gamma\geq0$ and $\gamma+\varrho\geq0$. 

On the other hand,
for $0<\tau<1$ and $\tau\leq 2-\frac{2}{\varrho_\Gamma}$, one can check that
\begin{equation}
	\begin{aligned}
		(\gamma,\cdots,\gamma,\gamma+\varrho)=\gamma (1,\cdots,1,1+\frac{\varrho}{\gamma})\in\bar\Gamma.
		\nonumber
	\end{aligned}
\end{equation}

\begin{remark}
	In \cite{yuan-PUE1} we have constructed   admissible metrics under the assumption  \eqref{tau-alpha-3}. Obviously,
	the condition \eqref{tau-alpha-4} is broader than \eqref{tau-alpha-3},
	which hence extends some of results in \cite{yuan-PUE1}.
\end{remark}


 \medskip
\section{Geometric optimal condition}
\label{sec1-topo}

For the prescribed curvature equation \eqref{main-equ1}, as we discussed in Subsection \ref{sec1-AOC} 
\begin{itemize}
	\item The equation  \eqref{main-equ1} has the form of \eqref{equ0-0}.
	
\item The equation	\eqref{main-equ1}   can be further reduced to the equation 
	of the form 
	\begin{equation}
		\label{equ1-Schouten}
		\begin{aligned}
			f(\lambda(-\tilde{g}^{-1} A_{\tilde{g}})) =\psi.
		\end{aligned}
	\end{equation}
 \begin{remark}
	Notice that in \eqref{equ1-Schouten}, $f$ is different from that of \eqref{main-equ1}, with  the same notation.
\end{remark}

 \item Even for $f=\sigma_1$,  in general one could not expect the solvability of 
 $${f}(\lambda(\tilde{g}^{-1} A_{\tilde{g}})) =\psi$$ 
 in the conformal class of complete admissible metrics, due to the obstruction to the existence of complete metrics with positive scalar curvature.

\end{itemize}  
In addition,  we prove that
\begin{itemize}
	
		 \item {\em Analytic optimal condition:} \eqref{assumption-4} is the optimal condition, under which the equation \eqref{equ0-0} is of uniform ellipticity according to Proposition \ref{key-lemma1} and Lemma \ref{lemma5.11}. 
	Consequently,   \eqref{tau-alpha-sharp} is  the sharp condition so that the equation \eqref{main-equ1} can be reduced to \eqref{equ1-Schouten} being of uniform ellipticity.

	\item {\em Uniform ellipticity $\Rightarrow$ Solvability}: The results presented in Subsection \ref{subsec1-geometric-conclusion} 
	 reveal that the uniform ellipticity implies the solvability of  \eqref{equ1-Schouten} (and thus of  \eqref{main-equ1}) in the conformal class of smooth  complete admissible metrics.

\end{itemize}

A natural question to raise  is as follows.
\begin{problem}
	Can we drop the uniform ellipticity assumption 
	in Theorem \ref{existence1-compact-2}?
\end{problem}

This section is devoted to answering this problem via presenting  obstructions.
More precisely, we show that the uniform ellipticity assumption \eqref{fully-uniform2} imposed in  
Theorem \ref{existence1-compact-2}  cannot be further dropped in general.
Also this shows that  assumption \eqref{tau-alpha-sharp} on $\tau$ and $\alpha$ is sharp.

	
	In prequel \cite{yuan-PUE1} the author solved the following equation when 
	$\Gamma\neq\Gamma_n$
	\begin{equation}
		\label{conformal-equ0}
		\begin{aligned}
			f(\lambda(\tilde{g}^{-1}G_{\tilde{g}}))=\psi.
	\end{aligned}	\end{equation}  
	Unlike the case $\Gamma\neq\Gamma_n$,  at least in dimension three, the equation
	\eqref{conformal-equ0} with $\Gamma=\Gamma_n$ is in general unsolvable in the conformal class of smooth complete metrics, 
	as shown by  a topological obstruction. 
	Nevertheless,  this topological obstruction relies crucially on \eqref{sectional-einstein}, the relation between sectional curvature and Einstein tensor which holds only in dimension three. 
	
	Below we remove the dimension restriction  
	via some obstruction from function theory. 
In addition we extend the topological obstruction 
 to 
 general equation \eqref{equ1-Schouten} with type 1 cone.
To this end, in this paper we prove  the following 
key ingredients.

\begin{lemma}\label{lemma1}	For any $\lambda\in\Gamma$, we have	\begin{equation}		\begin{aligned}		\lambda_{i_1}+\cdots +\lambda_{i_{\kappa_\Gamma+1}}>0, \mbox{   } \forall 1\leq i_1<\cdots<i_{\kappa_\Gamma+1}\leq n. \nonumber		\end{aligned}	\end{equation} \end{lemma}

\begin{proof}  
 The proof is based on the  convexity, symmetry and openness of $\Gamma$. 
\end{proof}

Specifically, we obtain
\begin{lemma}
	\label{lemma1-type1-yuan}
	Suppose $\Gamma$ is of type 1. Then for any
	$\lambda\in\Gamma$ there holds $$\sum_{j\neq i}\lambda_j>0, \,\, \forall 1\leq i\leq n. $$
	
\end{lemma}

As a consequence, we derive
\begin{proposition}
	\label{corollary-2}
Suppose $\Gamma$ is of type 1. 
	If $\lambda(-g^{-1}A_g)\in \Gamma$ then $G_g>0$.  
\end{proposition}



\subsubsection*{\bf Obstructions}
With this at hand,
we may give some obstructions from topology and function theory to show that, whenever the corresponding cone $\Gamma$ is of type 1  (equivalently $f$ obeys \eqref{fully-uniform2} in $\Gamma$ by Lemma \ref{lemma5.11}),   the equation
\eqref{equ1-Schouten} 
is in general  unsolvable in the conformal class of smooth complete  admissible metrics.

	Denote $B_{r}(a)=\left\{x\in \mathbb{R}^n: |x-a|^2<r^2\right\}.$ Suppose 
$\bar B_{r_1}(a_1), \cdots, \bar B_{r_m}(a_m)$ are pairwise disjoint.
Pick $r\gg 1$ so that $\cup_{i=1}^m B_{r_i}(a_i)\subset B_r(0)$. Then  we denote
$$\Omega=B_{r+1}(0)\setminus (\cup_{i=1}^m \bar B_{r_i}(a_i)).$$ 
On $\Omega$, one 	 may pick $g$ the Euclidean  metric, which is conformal to $\frac{\delta_{ij}dx^i\otimes dx^j}{((r+2)^2-|x|^2)^2}$.


	\begin{enumerate}
		\item {\em Topological obstruction}: Assume $n=3$.
		If Theorem  \ref{existence1-compact-2} holds 
		for 	$${f}(\lambda(-\tilde{g}^{-1}A_{\tilde{g}}))= \psi$$ 
		with  some  type 1 cone $\Gamma$,	
		according to Propositions \ref{prop1-einstein-sectional} and \ref{corollary-2},
		the solution  
		on $\Omega$  is then
		a complete 
		metric  
		with negative sectional curvature. 
		This contradicts to the  Cartan-Hadamard theorem.  
		
		\item {\em Obstruction  from  function theory}: Suppose that the following problem 
		\begin{equation}  f(\lambda(D^2 u))= \psi e^{u}   		\mbox{ in }\Omega, \,\, u=+\infty \mbox{ on } \partial \Omega \nonumber 	\end{equation} 
 with some type 1 cone $\Gamma$ (i.e. $0\leq\kappa_\Gamma\leq n-2$),
  has an admissible solution $u\in C^\infty(\Omega)$. According to Lemma \ref{lemma1},
		$u$ is a strictly  $(\kappa_\Gamma+1)$-plurisubharmonic function. 
		This is a contradiction by   some results of Harvey-Lawson \cite{Harvey2012Lawson-Adv,Harvey2013Lawson-IUMJ}. 
	\end{enumerate}
	

 The above obstructions
 reveal that the uniform ellipticity assumption \eqref{fully-uniform2} imposed in  
 Theorem \ref{existence1-compact-2}   cannot be further dropped.



 \medskip

\section{Further remarks on open symmetric convex cones}
\label{sec1-cones}


For purpose of verifying  \eqref{tau-alpha-4} or  \eqref{key-condition1-construct}, which is imposed as a proper condition
to construct admissible metrics without pseudo-admissible metric assumption, it seems necessary to estimate $\varrho_\Gamma$ and $\varrho_{\tilde{\Gamma}}$. 

For this reason, the  cones with $\varrho_{\tilde{\Gamma}}\geq 2$ are of particular interest.
Below we will prove some related results.
First we prove a key ingredient by projection. 
Let $\Gamma_\infty$ 
be as in \eqref{construct1-Gamma-infty} the projection of $\Gamma$ to the subspace of former $n-1$ subscripts. 
As noted by 
 \cite{CNS3}, when $\Gamma$ is of type 1, $\Gamma_\infty$ is an open symmetric convex cone in $\mathbb{R}^{n-1}$ and with $\Gamma_\infty\neq\mathbb{R}^{n-1}$.
From the construction of $\Gamma_\infty$, we can verify that 
\begin{lemma}
	\label{lemma1-preserving-kappa}
	If $\Gamma$ is of type 1 (if and only if $\kappa_\Gamma\leq n-2$), then 
	\begin{equation}
		\label{statement1}
		\begin{aligned}
			\kappa_{\Gamma_\infty}=\kappa_\Gamma,  \,\, \varrho_\Gamma\leq \varrho_{\Gamma_\infty}.  \nonumber
		\end{aligned}
	\end{equation}
	
\end{lemma}


This is a key ingredient for estimating $\varrho_\Gamma$ from above.
Inspired by this lemma, we will construct certain cones iteratively via projection. 
To do this we assume $\kappa_\Gamma\leq n-2$ 
and then obtain $\Gamma_\infty$. 
When $\kappa_\Gamma =\kappa_{\Gamma_\infty} \leq n-3$,  
i.e., $\Gamma_\infty$ is also of type 1, 
similarly we further construct a cone, denoted by $$\Gamma^\infty_{\mathbb{R}^{n-2}}\subset \mathbb{R}^{n-2},$$  which is the projection of $\Gamma_\infty$ to the subspace of former $n-2$ subscripts.
Accordingly, we can construct the cones by projection as follows:
\begin{equation}
	\label{construction-cones-infty}
	\begin{aligned}
		\Gamma^\infty_{\mathbb{R}^{n-3}}\subset \mathbb{R}^{n-3}, \cdots, \Gamma^\infty_{\mathbb{R}^{\kappa_\Gamma+1}} \subset \mathbb{R}^{\kappa_\Gamma+1}.
	\end{aligned}
\end{equation}
For simplicity, we denote $$\Gamma^\infty_{\mathbb{R}^{n-1}}:=\Gamma_\infty. 
$$ 

In fact for $\kappa_\Gamma+1\leq k\leq n-1$, one can check that 
\begin{equation}
	\label{construction2-cones-infty}
	\begin{aligned}
		\Gamma^\infty_{\mathbb{R}^{k}}
		=\left\{(\lambda_1,\cdots,\lambda_{k})\in \mathbb{R}^{k}: (\lambda_1,\cdots,\lambda_{k},  {\overbrace{R,\cdots,R}^{(n-k)-\mathrm{entries}}}) \in\Gamma  \mbox{ for some } R>0 
		\right\}.  
	\end{aligned}       
\end{equation}
This construction was also considered in prequel \cite[Section 3]{yuan-PUE1}. 

For $1\leq k\leq n$  we denote 
\begin{equation}
	\label{def-P_k}
	\mathcal{P}_{k}	=\{(\lambda_1,\cdots,\lambda_{n}): \lambda_{i_1}+\cdots+\lambda_{i_{k}}>0, \,  \forall 1\leq i_1<\cdots<i_{k}\leq n\}. 
\end{equation} 
First  we can check that

\begin{lemma}
	\label{lemma1-rigidity}
	If $\Gamma=\mathcal{P}_k$ for some $1\leq k\leq n$, then $\kappa_{\Gamma}=k-1$ and
	$\varrho_{\Gamma}=k. $
	In particular, $\varrho_\Gamma=1+\kappa_\Gamma.$
\end{lemma}

According to Lemma \ref{lemma1-preserving-kappa} we conclude the following proposition.
\begin{proposition}
	\label{lemma1-upper-rho}
For any  $\Gamma$, we have
$$\varrho_\Gamma\leq  \kappa_\Gamma+1,$$
	with equality if and only if
	$$\Gamma 
	=\mathcal{P}_{k} \mbox{ for some } 1\leq k\leq n.$$
	
	
\end{proposition}


\begin{proof}
For  $\kappa_\Gamma=n-1$, the statement is obvious.
Next we assume   $\kappa_\Gamma\leq n-2$.
	It suffices to prove rigidity. 
	If $\varrho_\Gamma=\kappa_\Gamma+1$, then $\varrho_{\Gamma^\infty_{\mathbb{R}^{\kappa_\Gamma+1}}} =\kappa_\Gamma+1$. This implies 
	\[\Gamma^\infty_{\mathbb{R}^{\kappa_\Gamma+1}}=
	\left\{(\lambda_1,\cdots,\lambda_{1+\kappa_\Gamma}): \sum_{j=1}^{\kappa_\Gamma+1}\lambda_j>0 \right\},\]
	and then $\Gamma=\mathcal{P}_{\kappa_\Gamma+1}$.

\end{proof}

As a consequence,   
we may verify  condition \eqref{tau-alpha-4} in some case.

\begin{corollary}
	\label{coro1-verify}
	Given a cone $\Gamma$ with $\kappa_\Gamma\leq n-3$, we have $1+(n-2)\varrho_\Gamma^{-1}\geq 2$.
\end{corollary}

Moreover, 
  together with Lemma 
   \ref{lemma1-rigidity} we may rewrite Lemma  \ref{lemma1}  as follows:
\begin{lemma}
	\label{lemma2}
	Suppose $\Gamma$ is a cone of $\kappa_{\Gamma}=k$ with $1\leq k\leq n-2$. Then 
	$\Gamma\subseteq\mathcal{P}_{k+1}.$
\end{lemma}

Then we  
revise
a question raised in \cite{yuan-PUE1}.

\begin{problem}
	Let $\Gamma$ be a cone of $\kappa_{\Gamma}=k$ with $1\leq k\leq n-2$. 
For such $\Gamma$, is $\Gamma_{n-k}\subseteq \Gamma$ correct? If so then 
	$\varrho_\Gamma  \geq 
	 \frac{n}{n-\kappa_{\Gamma}}.$
		
	
\end{problem}

From Subsection \ref{subsec1-construction-admissiblemetric} it is interesting to compute $\varrho_{\tilde{\Gamma}}$  or give a lower bound. 

\begin{proposition}	\label{lemma1-lower-rho}
	Given a cone $\Gamma$ and a constant $\varrho\leq \varrho_\Gamma$ with $\varrho\neq0$, we assume  $\tilde{\Gamma}$ is the corresponding cone as in \eqref{map1}. Then we have
	\begin{enumerate}
		\item If $\varrho<0$ then $\tilde{\Gamma}$ is of type 2 and $\varrho_{\tilde{\Gamma}}
		=\varrho_\Gamma+\frac{\varrho_\Gamma(n-\varrho_\Gamma)}{\varrho_\Gamma-\varrho}$.
		
		\item If $0<\varrho\leq\varrho_{\Gamma}$ then
		\begin{itemize}
			\item When $\Gamma$ is of type 1, $\varrho_{\tilde{\Gamma}}=n-\varrho.        $ In particular, $\varrho_{\tilde{\Gamma}}\geq n-\varrho_\Gamma\geq n-\kappa_\Gamma-1$. 
			
			\item When $\Gamma$ is of type 2, $\varrho_{\tilde{\Gamma}}>n-\varrho.$ 
			In particular, $\varrho_{\tilde{\Gamma}}+\varrho_\Gamma> n$.
		\end{itemize}
	\end{enumerate}
\end{proposition}

\begin{proof}
	
	Case 1: $\varrho<0$.
	Let $\mu=(1,\cdots,1,1-\varrho_\Gamma)\in\partial \Gamma$. The corresponding vector
	\begin{equation}
		\begin{aligned}
			\lambda=\frac{\varrho-\varrho_\Gamma}{\varrho}(1,\cdots,1,1-\frac{\varrho_\Gamma(n-\varrho)}{\varrho_\Gamma-\varrho})\in\partial\tilde{\Gamma}. \nonumber
		\end{aligned}
	\end{equation}
	
	Case 2: $0<\varrho\leq\varrho_\Gamma$. Set $\mu=(0,\cdots,0,1)\in\bar \Gamma$.   Accordingly we have
	\begin{equation}
		\begin{aligned}
			\lambda=\frac{1}{\varrho}(1,\cdots,1, 1-(n-\varrho))\in\overline{\tilde{\Gamma}}. \nonumber
		\end{aligned}
	\end{equation}
	
	
	If $\Gamma$ is of type 1, then $\varrho_{\tilde{\Gamma}}=n-\varrho.$
	So  $\varrho_{\tilde{\Gamma}}\geq n-\varrho_\Gamma\geq n-\kappa_\Gamma-1$ by Proposition \ref{lemma1-upper-rho}.  
	In particular, when $\tilde{\Gamma}$ is of type 2 (if and only if $\varrho<\varrho_\Gamma$), we obtain  $	\varrho_{\tilde{\Gamma}}>n-\varrho. $ 
	
	If $\Gamma$ is of type 2, then $	\varrho_{\tilde{\Gamma}}>n-\varrho. $  
	
\end{proof}

\begin{corollary}
	Fix a cone $\Gamma$.
	Let  $\varrho$ be as in  \eqref{beta-gamma-A-3}, and assume $\varrho\leq \varrho_\Gamma$.  
As in  \eqref{map1} we obtain $\tilde{\Gamma}$. We have $\varrho_{\tilde{\Gamma}}\geq 2$, provided either one of the following hods.
	\begin{enumerate}
		\item 	$(\alpha,\tau)$ obeys \eqref{tau-alpha-4}.
		
		\item 
		$\tau>1$ and 
		$\kappa_\Gamma\leq n-3$.
	\end{enumerate}
	
\end{corollary}

However, 
one  could not expect that one can obtain  an effective estimate for all type 2 cones as shown by the following:


\begin{corollary}
	\label{corollary1-example-type2}
	For  any   $1<t<n$, there is a type 2 cone $\tilde{\Gamma}$
	with  $\varrho_{\tilde{\Gamma}}=t.$ 
\end{corollary}
\begin{proof}
	Given $1<t<n$, there is a unique $\varrho<0$ so that $\frac{ n-\varrho}{1-\varrho}=t$. For such $\varrho$  we get a type 2 cone
	$$\tilde{\Gamma}=\left\{(\lambda_1,\cdots,\lambda_n)\in\mathbb{R}^n: \sum_{j=1}^n \lambda_j -\varrho\lambda_i >0, \, 1\leq i\leq n\right\}.$$ 
	By Proposition \ref{lemma1-lower-rho}, $\varrho_{\tilde{\Gamma}} 
	=1+\frac{ n-1}{1-\varrho}=t$. 
	
\end{proof}

 \medskip
 \subsubsection*{Acknowledgements} 
 The author  wishes to express his gratitude to Professor Yi Liu for 
 answering questions related to the proof of Lemma \ref{lemma-diff-topologuy}.
The author also  wishes to   thank Ze Zhou for useful discussion on the homogeneity lemma.


 \bigskip
 

\end{document}